\documentclass[12pt]{amsart}

\usepackage{graphicx, verbatim,amsmath,amssymb, float}
\usepackage{hyperref}
\hypersetup{
    colorlinks=true,
    linkcolor=blue,
    filecolor=magenta,      
    urlcolor=cyan,
}

\usepackage{amsthm}
\usepackage{amsfonts}
\usepackage{cleveref}
\usepackage{comment}
\usepackage{braket}

\newcommand{\tdens}{\tau}

\newcommand{\ER}{Erd\H{o}s-R\'enyi}

\newcommand{\R}{\mathbb{R}}

\newcommand{\calG}{\mathcal{G}}

\newcommand{\calP}{\mathcal{P}}
\newcommand{\be}{\begin{equation}}
\newcommand{\ee}{\end{equation}}
\newcommand{\www}{{\mathcal W}}

\newtheorem{lemma}{Lemma}

\newtheorem{theorem}[lemma]{Theorem}
\newtheorem{corollary}[lemma]{Corollary}

\title{Optimal graphons in the edge-2star model}
\author{Charles Radin and Lorenzo Sadun}
\address{Charles Radin\\Department of Mathematics\\The University of
  Texas at Austin\\ Austin, TX 78712} \email{radin@math.utexas.edu}
\address{Lorenzo Sadun\\Department of Mathematics\\The University of
  Texas at Austin\\ Austin, TX 78712} \email{sadun@math.utexas.edu}

\date{\today}

\begin{document}
\begin{abstract}
  In the edge-2star model with hard constraints we prove the existence
  of an open set of constraint parameters, bisected by a
  line segment on which there are nonunique entropy-optimal graphons
  related by a
  symmetry. At each point in the open set but off the line segment there is a unique
  entropy-optimizer, bipodal and varying analytically with the constraints. 
  We also show that throughout another open set, containing a different portion
  of the same line of symmetry, there is instead a unique optimal
  graphon, varying analytically with the parameters. 
  We explore the extent
  of these open sets, determining the point at which a symmetric graphon ceases to be
  a local maximizer of the entropy. Finally, we prove some foundational theorems in 
  a general setting, relating optimal graphons to the Boltzmann entropy and the 
  generic structure of large constrained random graphs. 
 \end{abstract}
\maketitle

\section{Introduction}

This paper serves two purposes, the primary one being
to derive results about the edge-2star graphon
model in which we consider large dense random graphs 
with hard constraints
on the density $e$ of edges and $t$ of 2stars. (A 2star, sometimes called a 
``cherry'', is a simple graph with three vertices and two edges.) This is the 
simplest model
in which we employ hard competing constraints, allowing for strong 
rigorous results about non-constant graphons, or equivalently about large graphs that
are not \ER{}. 

The second purpose of the paper is to provide proofs of two theorems relating Boltzmann
entropy and ``typical'' large graphs to solutions of an optimization problem on 
graphons. These are described in more detail below, but the gist of
the first, originally proven in less generality in \cite{RS1,RS2}, is that
the Boltzmann entropy associated with some constrained subgraph densities, which is 
the rate at which the number of graphs with those densities grows with the number 
of vertices, is the same as the maximal Shannon entropy of a graphon meeting certain
integral constraints. The second theorem says that, if the graphon optimization problem
has a unique solution, then all but exponentially few large graphs with the 
specified subgraph densities have a structure very close to that 
described by the optimal graphon. Taken together, they imply that solving problems 
involving hard constraints on graph{ons} is tantamount to understanding 
the ensemble of  large constrained random {graphs}.

\subsection{Results about the edge-2star model}

Our first result on the edge-2star model concerns the region with edge density 
close to $1/2$ and 2star density close to the maximum. (See Figure \ref{Fig-L1}.)
 We show that
the open set of 
graphons with
reduced 2star density $\tilde t = t - e^2$ close to the maximum has
two open subsets,
one of which we call ``clique-like'' and the other of which we call
``anti-clique-like'', separated by a segment of the line $e=\frac12$. 
(See Figure \ref{Fig-L2}.)
\begin{theorem}[Theorems \ref{thmCliAnticli},\ref{thm2}] \label{main1} 
There is an open set in $(e, \tilde t)$-space containing a segment of
the line $e=\frac12$ with $\tilde t$ just below its maximum of $\frac{\sqrt{2}-1}{4}$,
such that
\begin{itemize} 
\item When $e>\frac12$, the entropy-optimizing graphon is unique and clique-like,  
\item When $e < \frac12$, the entropy-optimizing graphon is unique and anti-clique-like, and  
\item When $e=\frac12$ there are two entropy-optimizing graphons, one
clique-like and one anti-clique-like.
\end{itemize}
\end{theorem}
We thus show that, on each side of a segment of the line $e=1/2$, the optimal
graphon is bipodal and unique, with parameters that vary smoothly with
$e$ and $\tilde t$.  This implies that there is a {\em discontinuous}
phase transition across the line $e=\frac12$, with typical graphs
being anti-clique-like on one side of the line and clique-like on the
other.

The situation is very different when $\tilde t$ is small. Let 
$\zeta = \sqrt{\left (e-\frac12 \right )^2 + \tilde t}$. 

\begin{theorem}[Theorem \ref{thm:low-t}]\label{main2} 
For sufficiently small $\zeta$, the entropy-maximizing graphon is 
unique and bipodal, with parameters
\begin{eqnarray}
a & = & 1 - e - 2\zeta + O(\zeta^2), \cr 
b & = & 1-e + 2 \zeta + O(\zeta^2), \cr
d & = & 1-e + O(\zeta^2), \cr 
c & = & \frac12 \left (1 - \frac{e-\frac12}{\zeta} \right ) + O(\zeta)
\end{eqnarray}
that are analytic functions of $(e,\tilde t)$ everywhere except at the singular point 
$e=\frac12$, $\tilde t = 0$. 
\end{theorem}

In previous work \cite{KRRS1} we had proven that, for $e \ne 1/2$ and $\tilde t$
sufficiently small, there is a unique optimal graphon that is bipodal. Theorem
\ref{main2} bridges the gap between the regions $e<1/2$ and $e>1/2$ and shows that 
there is a single phase just above the entire \ER{} curve $\tilde t = 0$. 

Now consider what happens along the line $e=1/2$. When $\tilde t$ is small, Theorem
\ref{main2} implies that there is a unique optimal graphon that is bipodal. 
Uniqueness implies that the four parameters satisfy $a+b=1$ and $c=d=1/2$. At some
point on the line a graphon of this form ceases to be optimal, since 
Theorem \ref{main1} says that,
for $e = 1/2$ and $\tilde t$ sufficiently large, there are two optimal graphons,
one clique-like and the other anti-clique-like. At some point in between there must 
be a point of non-analyticity, where symmetry is broken and 
the structure of the optimal graphon changes.

Determining the exact nature of this bifurcation point is beyond our current methods, as
it is conceivable that the optimal graphon might change discontinuously. Instead, we
determine where the symmetric graphon becomes stable 
against {\em small} changes.

\begin{theorem}[Theorem \ref{thm:bifurcation}] \label{main3} 
There is a number $\tilde t^* \approx 0.03727637$ such that
\begin{enumerate} 
\item For all $\tilde t < \tilde t^*$, there is a bipodal graphon with $b=1-a$, $c=d=1/2$, 
that is a local maximizer of the entropy among all bipodal graphons with edge density 
$1/2$ and 2star density $\tilde t + \frac14$.
\item If $\tilde t^* < \tilde t \le 0.0625$, then there exist bipodal graphons with $b=1-a$ and
$c=d=1/2$, but these graphons are not local maximizers. 
\item If $\tilde t > 0.0625$, then there do not exist bipodal graphons with $b=1-a$ and $c=d=1/2$. 
\end{enumerate}
\end{theorem}

\subsection{Background and formalism} 
To put these results about the edge-2star model in context, 
and to explain our foundational results, we review 
some relevant history of research into
ensembles of large dense random graphs. 

Following the publication by
Chatterjee/Varadhan \cite{CV} of the LDP of the Erd\H{o}s-R\'enyi
random graph $G(n,p)$, Chatterjee/Diaconis popularized \cite{CD} the
use of graphons, with `soft' constraints on the densities of several
subgraphs $P_j$, to analyze  exponential
random graph models (ERGMs).
The graphon formalism of Lov\'asz and coauthors
\cite{BCL,BCLSV,LS1,LS2,LS3} allows graphs $G$ on any finite number of
nodes to be incorporated, as `checkerboard graphons' $g^G$, in the
space $\www$ of their `infinite node limits', graphons; for an
in-depth presentation we recommend \cite{Lov}.

The LDP is
expressed in terms of probability distributions $\calP_n$ on $\www$ (and
the closely connected $\widetilde {\calP_n}$ on reduced graphons
$\widetilde \www$), associated with a sequence $p_n$ of discrete
distributions on the sets $\calG_n$ of graphs on $n$ nodes.
Given their focus on ERGMs, in \cite{CD} the constraints on the
subgraphs $P_j$ were naturally implemented by the choice of exponential
distributions for $p_n$:
\begin{equation}\label{eq:distribution on finite graphs}
  p_n(G)=e^{n^2 T({g^G})-\psi_n},
\end{equation}
where $\psi_n$ is a normalizing constant,
\begin{equation}T(g)=\sum_{j=1}^k \beta_j \tdens_j(g)
\end{equation}
is a function on graphons, 
and $\tdens_j(g)$ is the density of $P_j$ in the graphon $g$.  The
parameters of the model are the $\beta's$, and the constrained
graphons for given parameter values are the graphons $g$ that
optimize the functional $T(g)+S(g)$, where 
\begin{equation}
S(g) := \int_0^1\!\!\int_0^1 H(g(x,y)) dx \, dy \qquad \hbox{ and } H(u) := -\frac12
(u \ln(u) + (1-u)\ln(1-u)).\end{equation}
The quantity $S(g)$, which we call the Shannon entropy of the graphon $g$, is
closely related to the LDP rate function $\tilde I_p(g)$ of \cite{CD}.
Specifically, $S(g) = \frac12 \ln(2) - \tilde I_{1/2}(g)$.

Both \cite{CV} and \cite{CD} emphasize the difficulty of
accessing/determining nonconstant optimal graphons; the formalism 
easily leads to Erd\H{o}s-R\'enyi optima. (See for instance the open
questions section 4.8, in \cite{CV}.) It was to overcome this
tendency that a variant graphon model was introduced in \cite{RS1,RS2}
using hard rather than soft constraints on the subgraphs $P_j$; the
parameters were chosen to be the densities of the $P_j$
and the role of the discrete distribution $p_n$ on $\calG_n$ was
replaced by a two-step process. Then the appropriately
constrained graphons for given values of the parameters are
characterized as those $g$, with the given parameter densities, which
optimize $I(g)$. (See \cite{DL} for a connection to large
deviations for $G(n,m)$.)

This modified approach to parametric graphon models achieved the
initial goal of \cite{RS1,RS2}, the determination of a fully explicit
{\em nonconstant and unique} optimizer for each constraint on a line
in the edge-triangle model. The goal then expanded.  In \cite{CD}
(indeed already in \cite{CV}) attention was drawn to singular behavior
(`phase transitions') that appeared as the model parameters were
varied. Our extended goal was to determine a `phase', an open set of
parameters, with a unique, optimizing graphon associated to each
point, which moreover responds smoothly with variation of the
parameters. (Note that Erd\H{o}s-R\'enyi graphs are automatically
represented by constraint parameters on a curve in parameter space,
and from smoothness cannot be contained in a phase.) This took a few years to
accomplish but was obtained \cite{KRRS2} in a broad class of models:
constraints on edges and any one other graph, $P$. Finally, after
another few years, we determined \cite{NRS3} a `transition', a pair of
 phases separated by a transition curve. We emphasize that
in these models with hard constraints such transitions represent sharp
structural changes in the `typical' large graph as the parameters
vary, where typical means all but exponentially few as the node number
diverges.

An important lesson learned was that, as in the more general subject
of deviations in $G(n,p)$ from which this all stems \cite{CV}, in
analyzing our deviations it is significant whether we are dealing with
an upper tail or lower tail; for a model with fixed constraints on the
density of edges and one other graph $P$, it is significant whether
the density for $P$ is larger or smaller than it is for
Erd\H{o}s-R\'enyi graphs with the same edge density. (This is a very
large subject; for a good overview we recommend \cite{Ch}.
For a particularly relevant connection to this paper
see \cite{NRS2}, and references therein.) To prove the more detailed
results such as phase transitions required focusing on a narrower
range of models, edge-triangle for lower tail features and edge-2star
for upper tail. In the edge-triangle model we recently determined
\cite{NRS4} a `symmetric' phase which could be distinguished by an
order parameter, and, in the present paper, in the edge-2star model we
determine a discontinuous transition.

\subsection{Foundational results}

Our  goal is to analyze large graphs with hard constraints on the
densities of a number of subgraphs, typically the density $e$ of edges and 
$t_P$ of another subgraph $P$. If $P$ has $m$ vertices and $\ell$ edges, with 
edge $e_k$ connecting vertices $s_k$ and $f_k$, then the density of $P$ associated
with a graphon $g$ is given by the functional
\begin{equation}
\tdens(g) := \int \prod_{k=1}^\ell g(x_{s_k},x_{f_k}) d^mx. 
\end{equation}
If we are considering multiple subgraphs $P_i$, then we will refer to the 
density function for $P_i$ as $\tdens_i(g)$ and a typical value of this functional
as $t_i$. 

The key tool for counting finite graphs with subgraph densities in a given range is 
the LDP of Chatterjee and Varadhan \cite{CV}:
\begin{theorem}\label{thmCV} For any closed set $\tilde{F}
  \subseteq \widetilde{\www}$, and using the notation $|A^n|$ for the
  number of graphs on $n$ nodes whose checkerboard graphons $g^G$
lie in $A$, we have
  \begin{equation}{\lim \sup}_{n\to \infty} \frac{1}{n^2}
    \ln(|\tilde{F}|)\le {\sup}_{\tilde{g} \in \tilde{F}} \ S(\tilde{g})
\end{equation}
and for any open set $\tilde{U}\subseteq \widetilde{\www}$,
\begin{equation}
  {\lim \inf}_{n\to \infty} \frac{1}{n^2}\ln(|\tilde{U}|)\ge
  {\sup}_{\tilde{g}\in \tilde{U}} \ S(\tilde{g}).
 \end{equation} 
  \end{theorem}
To apply this theorem to graphs with constraints on subgraphs $P_1, \ldots, P_k$, 
we merely take $\tilde F$ and $\tilde U$ to be sets of graphons whose densities 
$(\tdens_1(g), \ldots, \tdens_k(g))$ lie in open and closed subsets of $\R^k$. 
We say that a collection $(t_1,\ldots,t_k)$ is {\em achievable} if there exists at least
one graphon $g$ with $(\tdens_1(g),\ldots,\tdens_k(g)) = (t_1,\ldots,t_k)$.

Next we define the Boltzmann entropy. If we are constraining the densities of 
$k$ subgraphs $P_1, \ldots P_k$, let $\displaystyle Z^{n,\delta}_{t_1,\ldots,t_k}$ 
be the number of simple graphs $G$
such that the density of each $P_i$ is in the interval $(t_i-\delta,t_i+\delta)$.
Let $\displaystyle B^{n,\delta}_{t_1,\ldots,t_k}={\ln(Z^{n,\delta}_{t_1,\ldots,t_k})}/{n^2}$ and
consider
\begin{equation}\label{Blimit}
  \lim_{\delta\to 0^+}\lim_{n\to \infty}B^{n,\delta}_{t_1,\ldots,t_k}.
  \end{equation}

  The double limit exists, defining $B_{t_1,\ldots,t_k}$, and there is a
  variational characterization of it, proven using the LDP. The following 
  is a generalization of results proven in \cite{RS1, RS2} (first for
  edges and triangles, then for edges and one other subgraph). 

\begin{theorem}\label{thmBoltzmann}
  For any achievable $k$-tuple $(t_1,\ldots,t_k)$, the limit (\ref{Blimit})
  defining $B_{t_1,\ldots, t_k}$ exists and equals $\max S(g)$, where the
  maximum is over all graphons $g$ with $(\tdens_1(g),\ldots,\tdens_k(g)) = (t_1,\ldots,t_k)$.
\end{theorem}

The (constrained) graphon $g$ that maximizes $S(g)$ doesn't just determine the {\em number}
of large graphs with subgraph densities close to $t_1, \ldots, t_k$. When the optimal 
graphon is unique, it also determines the form of all but an exponentially small 
fraction of those graphs. The following theorem states precisely what we
 mean when we say that a typical large graph with densities $(t_1,\ldots,t_k)$
looks like $g_0$.

\begin{theorem}\label{thmTypicality}
Let $(t_1,\ldots,t_k)$ be a point in the space of achievable parameter values in a
model with $k$ constrained subgraphs, and suppose that there
is a unique (reduced) graphon
$g_0$ that maximizes $S(g)$ subject to the constraint
$(\tdens_1(g),\ldots,\tdens_k(g)) = (t_1,\ldots,t_k)$.
For any positive constants $\delta$ and  $n$, let
$\mathcal{G}_{\delta,n}$ be the set of
labeled graphs on $n$ vertices with densities $\tdens_i(g)$ in $(t_i-\delta,
t_i+\delta)$ for each $i$. Then, for any $\epsilon>0$,
there exist
positive constants $\delta$, $N$ and $K$ such that, for all $n>N$, the fraction of
graphs in $\mathcal{G}_{\delta,n}$ that are within $\epsilon$ of $g_0$
in the cut metric
exceeds $1 - e^{-Kn^2}$.
\end{theorem}

Note that if $S(g_0)=0$, then the number of graphs in $\mathcal{G}_{\delta,n}$ 
for small $\delta$ grows slower than $e^{Kn^2}$. Theorem \ref{thmTypicality}
then implies that, for $\delta$ sufficiently small, {\em all\/} graphs in 
$\mathcal{G}_{\delta,n}$ are within $\epsilon$ of $g_0$.

The organization of this paper is as follows. In Section \ref{sec:old} we review what
has been previously proven about the edge-2star model. In Section \ref{sec:large} 
we consider the situation where the edge density $e$ is close to $1/2$ and the 2star 
density $t$ is close to its maximum and prove Theorem \ref{main1}. In Section 
\ref{sec:small} we study a neighborhood of $(e,\tilde t) = (1/2, 0)$ and prove 
Theorem \ref{main2}. In Section \ref{sec:bifurcation} we study the stability of 
the graphons found in Section \ref{sec:small} and prove Theorem \ref{main3}. 
Finally, 
in the Appendix we prove the foundational Theorems \ref{thmBoltzmann} and 
\ref{thmTypicality}.

\section{Old results about the edge-2star model} \label{sec:old}

In this section we review some facts about optimal graphons in the edge-2star model.
For detailed proofs, see \cite{KRRS1}.

Let 
\[ d(x)  = \int_0^1 g(x,y) dy \]
be the degree function of the graphon $g$. The 2star density is then 
\[ t = \int_0^1 d(x)^2 \, dx. \]
We also define the reduced 2star density 
\[ \tilde t = t-e^2 = \int_0^1 (d(x)-e)^2 \, dx. \]
The minimum value of $\tilde t$ is obviously zero, and is achieved when 
$d(x)$ is constant.
Among constant-degree graphons with edge density $e$, 
the entropy maximizer is the (constant) 
Erd\H{o}s-R\'enyi graphon $g(x,y)=e$.

The maximum value of $t$ depends on $e$. 
When $e > 1/2$, the maximum value of $t$ is $e^{3/2}$ and 
is achieved by a clique. This is a graphon that is 1 on a square $I \times I$, 
where $I$ is an interval of width $\sqrt{e}$, and is zero everywhere else. 
When $e < 1/2$, however, the 
maximum value of $t$ is $(1-e)^{3/2} + 2e-1$ and is achieved by an anti-clique. This is a graphon
that is equal to 0 on a square of side $\sqrt{1-e}$ and is 1 everywhere else. 
When $e=1/2$, the maximum value of $t$ is $\sqrt{2}/4 \approx 0.3535$ and is achieved by either a clique or an anti-clique, in either case with the interval having 
width $\sqrt{2}/2$. 

\begin{figure}
\center{
\includegraphics[width=3in]{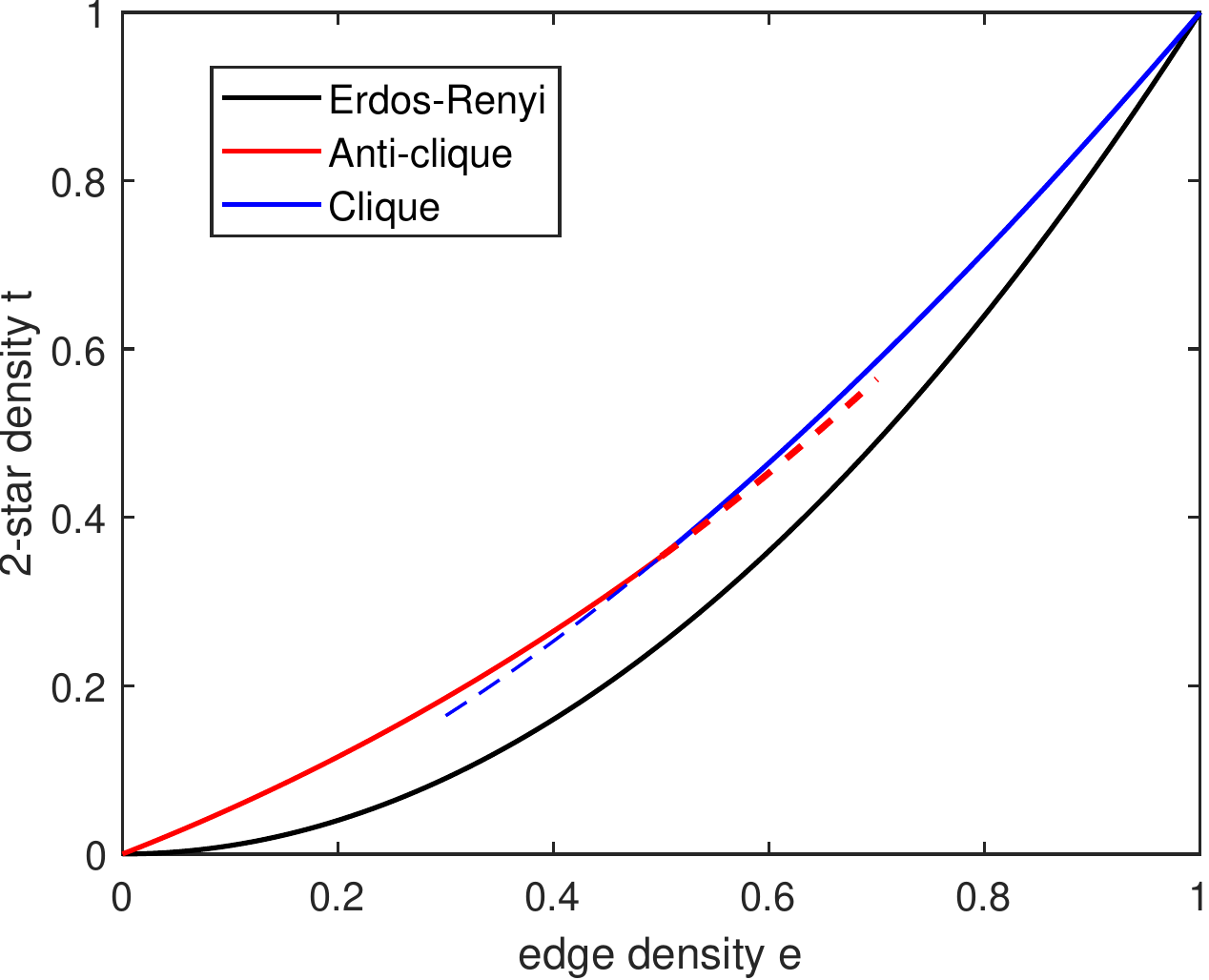}}
\caption{Possible values of $(e,t)$ in the edge-2star model.}\label{Fig-L1}
\end{figure}

If we replace a graphon $g$ with $1-g$, then this changes $e$ to $1-e$, but does not change $\tilde t$ or the entropy $S$. Applying the symmetry to an optimal graphon for given
values $(e, \tilde t)$ gives an optimal graphon with values $(1-e,\tilde t)$.
The possible values of the
edge and 2star densities are more cleanly expressed in terms of $\tilde t$ rather than
$t$, as in Figure \ref{Fig-L2}.

\begin{figure}
\center{
\includegraphics[width=3in]{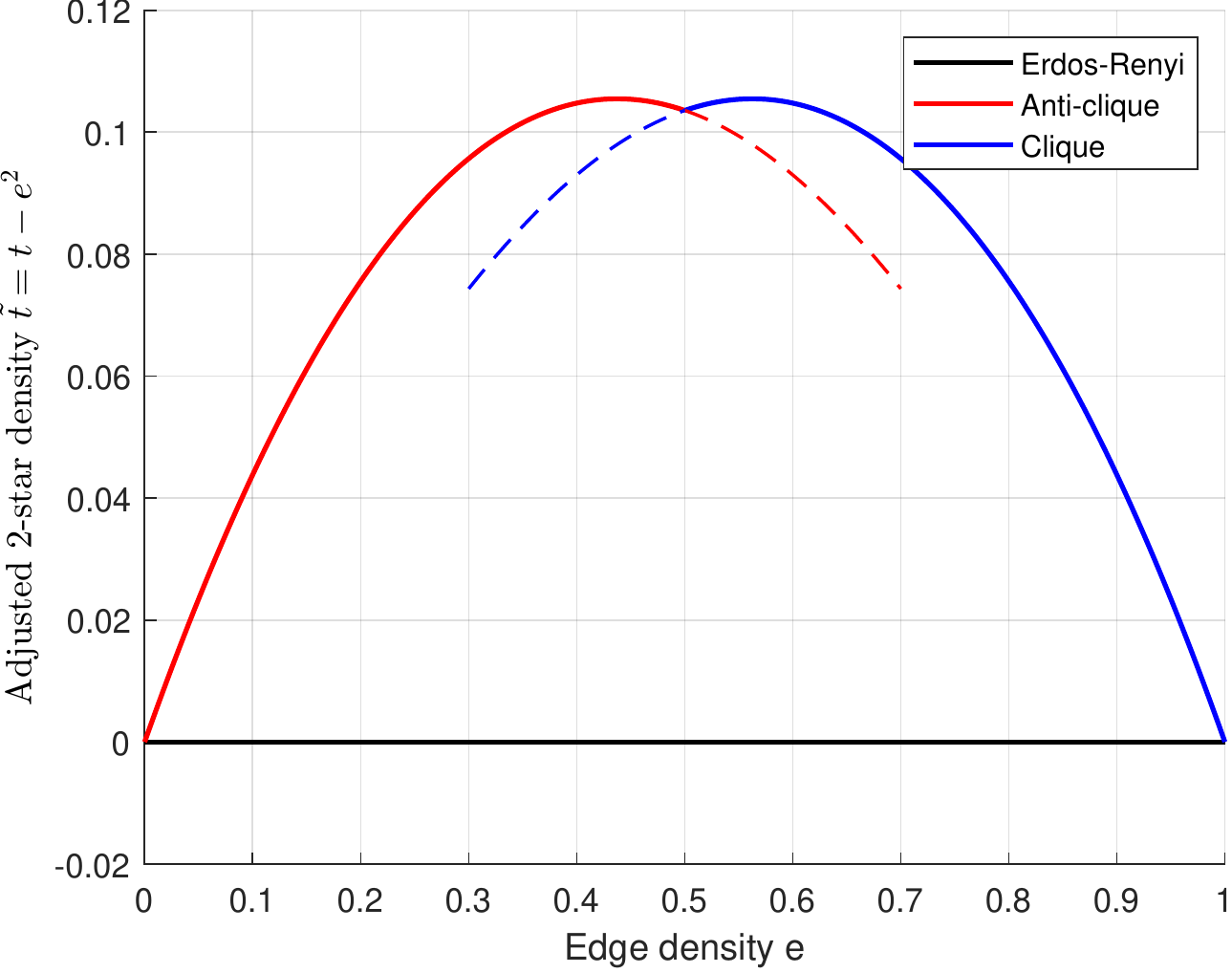}}
\caption{Possible values of $(e,\tilde t)$ in the edge-2star model.}\label{Fig-L2}
\end{figure}

At a stationary point of the entropy, the degree 
function $d(x)$ determines the graphon via the equation 
\begin{equation}\label{eq1} 
g(x,y) = \frac{1}{1 + \exp(-(\alpha+\beta (d(x)+d(y))))},
\end{equation} 
where $\alpha$ and $\beta$ are Lagrange multipliers, with $dS = \alpha de + \beta dt$ as we vary the graphon in arbitrary ways.
By integrating over $y$ we get the self-consistency equation 
\begin{equation}\label{eq2} 
d(x) = \int_0^1 \frac{dy} {1 + \exp(-(\alpha+\beta (d(x)+d(y))))}. 
\end{equation}
That is, the only possible values of $d(x)$ are solutions to the equation 
\begin{equation}\label{eq3} 
z = k(z) : = \int_0^1 \frac{dy} {1 + \exp(-(\alpha+\beta (z+d(y))))}. 
\end{equation}
Both sides of equation (\ref{eq3}) are analytic functions of $z$, so there can only be a finite number of solutions. This implies that all graphons that are 
stationary points of the constrained entropy functional are multipodal. 

\section{Optimal graphons when $\tilde t$ is large}\label{sec:large}

We say that a graphon is clique-like if its degree function is $L^1$-close to a step function with values $\sqrt{e}$ and 0, 
and is anti-clique-like if its degree function is $L^1$-close to a step function with values $1-\sqrt{1-e}$ and 1. As we approach
the upper boundary, all graphons must be clique-like or anti-clique-like, since otherwise we could take a limit as $t$ approaches
the maximum and get a $t$-maximizing graphon that isn't a clique or an anti-clique. 
In particular, all of the entropy-maximizing
graphons in a neighborhood of $(e, \tilde t) = \left (\frac12, \frac14(\sqrt{2}-1)\right )$ must be clique-like or anti-clique-like. 
The two sets are related by the $g \leftrightarrow 1-g$ symmetry, so it is 
sufficient to study clique-like graphons. 

Let $g$ be a clique-like graphon that is a stationary point of the entropy. If we increase the size $c \approx \sqrt{e}$ of the pode(s) with 
degree function close to $\sqrt{e}$ at the expense of those with degree function close to $0$, then we do not change the 
set of values achieved by $H(g(x,y))$. We only change the area of the regions where each value is achieved. This means that
the change in the entropy (per change in $e$ or $t$) is bounded by a multiple of the existing entropy, which goes to zero
as we approach the upper boundary. That is, with this move we must have
\begin{equation} \alpha \frac{de}{dc} + \beta \frac{dt}{dc} = \frac{dS}{dc} = o(1). \end{equation}
It is easy to check that $\frac{dt}{dc} \approx \frac32 \sqrt{e} \frac{de}{dc}$, so 
\begin{equation}\label{a-over-b} \frac{\alpha}{\beta} = -\frac32 \sqrt{e} + o(1) \end{equation}
Note that $\beta$ is negative, as the entropy decreases as we approach the upper boundary, so $\alpha$ is positive. Both parameters
diverge as we approach the upper boundary. 

\begin{theorem}\label{thmCliAnticli} There is an open set $U$ in $(e,
  \tilde t)$-space containing a segment of
the line $e=\frac12$ with $\tilde t$ just below its maximum of $\frac{\sqrt{2}-1}{4}$,
such that
\begin{itemize} 
\item When $e>\frac12$, the entropy-optimizing graphon is clique-like, and 
\item When $e < \frac12$, the entropy-optimizing graphon is anti-clique-like.  
\end{itemize}
\end{theorem}

\begin{proof} Thanks to the $g \leftrightarrow 1-g$ symmetry that changes $e$ to $1-e$ and 
swaps clique-like and anti-clique-like graphons, the second statement is equivalent to the
first, so it is sufficient to prove the first. We henceforth assume that $e>1/2$ and 
that both clique-like and anti-clique-like graphons exist with densities $(e, \tilde t)$.

Let $S_c(e, \tilde t)$ be the maximum entropy achievable by a clique-like graphon and let $S_a(e,\tilde t)$ be the maximum
achievable by an anti-clique-like graphon. Thanks to our $g \leftrightarrow 1-g$ 
symmetry,
\begin{equation}
S_c(e,\tilde t) = S_a(1-e, \tilde t),
\end{equation}  
and in particular $S_c(\frac12, \tilde t) = S_a(\frac12, \tilde t)$. 
If we have the optimal clique-like graphon and move along a line of constant $\tilde t$, 
then $dt = 2e de$, so 
the change in entropy is proportional to $(\alpha +2 e \beta) de$. By equation 
(\ref{a-over-b}), 
\begin{equation} \alpha + 2e \beta \approx \left ( 2e - \frac32 \sqrt{e} \right ) \beta
\approx \left ( 1 - \frac{3\sqrt{2}}{4} \right ) \beta. \end{equation} 
Since $\beta$ is negative and $3 \sqrt{2}/4 > 1$, this quantity is positive, making 
$S_c(e,\tilde t)$ an increasing function of $e$.  
In particular,
\begin{equation}
S_c(e, \tilde t) > S_c(1-e, \tilde t) = S_a(e,\tilde t).
\end{equation} 
That is, the best clique-like graphon has a higher value of $S$ 
than the best anti-clique-like graphon, so the best overall graphon is clique-like.  
\end{proof}

\begin{theorem}\label{thm2}
On the subset of $U$ where $e \ge 1/2$, the optimal clique-like graphon is unique and bipodal. 
\end{theorem}

Combined with Theorem \ref{thmCliAnticli}, this says that there is a unique entropy-maximizing graphon when $e>1/2$, and that this 
optimizing graphon is clique-like and bipodal. By the $g \to 1-g$ symmetry, there is a unique entropy-maximizing graphon when
$e < 1/2$, and that graphon is anti-clique-like and bipodal. When $e=1/2$, there are exactly two optimizing graphons, both 
bipodal, one clique-like and one anti-clique-like.

\begin{proof} First note that optimal graphons must exist for each $(e,\tilde t)$, 
thanks to the compactness of the space of reduced graphons and the semi-continuity of the 
Shannon entropy functional $S(g)$. With that in mind,
suppose that $g$ is an optimal clique-like graphon. We will prove properties of $g$ in stages:
\begin{enumerate}
\item The degree function $d(x)$ only takes values close to 0, $\frac12 \sqrt{e}$, or $\sqrt{e}$. 
\item The degree function $d(x)$ only takes values close to 0 or $\sqrt{e}$. 
\item The degree function $d(x)$ only takes two values, one close to 0 and one close to $\sqrt{e}$. That is, $g$ is bipodal.
\item The parameters that define this bipodal graphon are uniquely determined. 
\end{enumerate}

Plugging equation (\ref{a-over-b}) into equations 
(\ref{eq1} --\ref{eq3}) gives 
\begin{eqnarray} \label{3eqs}
g(x,y) &=& \frac{1}{1 + \exp(-\beta(d(x)+d(y) - \frac32\sqrt{e} + o(1)))} \cr 
d(x) & = & \int_0^1 \frac{dy}{1 + \exp(-\beta(d(x)+d(y) - \frac32\sqrt{e} + o(1)))} \cr 
k(z) & = & \int_0^1 \frac{dy}{1 + \exp(-\beta(z+d(y) - \frac32\sqrt{e} + o(1)))}.
\end{eqnarray}

Since $\beta$ is large, the function $\frac{1}{1 + \exp(-\beta(d(x)+d(y) - \frac32\sqrt{e} + o(1)))}$ is close to 1 whenever 
$d(x)+d(y)$ is bigger than $\frac32 \sqrt{e}$, is close to 0 whenever $d(x)+d(y)$ is smaller than $\frac32 \sqrt{e}$, and 
only takes values substantially different from 0 or 1 when $d(x)+d(y)$ is very close to a fixed threshold value that is 
$\frac32 \sqrt{e} + o(1)$. Note that $\sqrt{e} \approx \sqrt{2}/2 > 2/3$, so the threshold is greater than 1. Since the function
$d(y)$ is close to 0 on a set of measure approximately $1-\sqrt{e}$ and close to $\sqrt{e}$ on a set of measure approximately
$\sqrt{e}$, and only takes on other values on sets of small measure, the function $k(z)$
is approximately a step function, as shown in Figure 3. 

\begin{figure}
\center{
\includegraphics[width=3in]{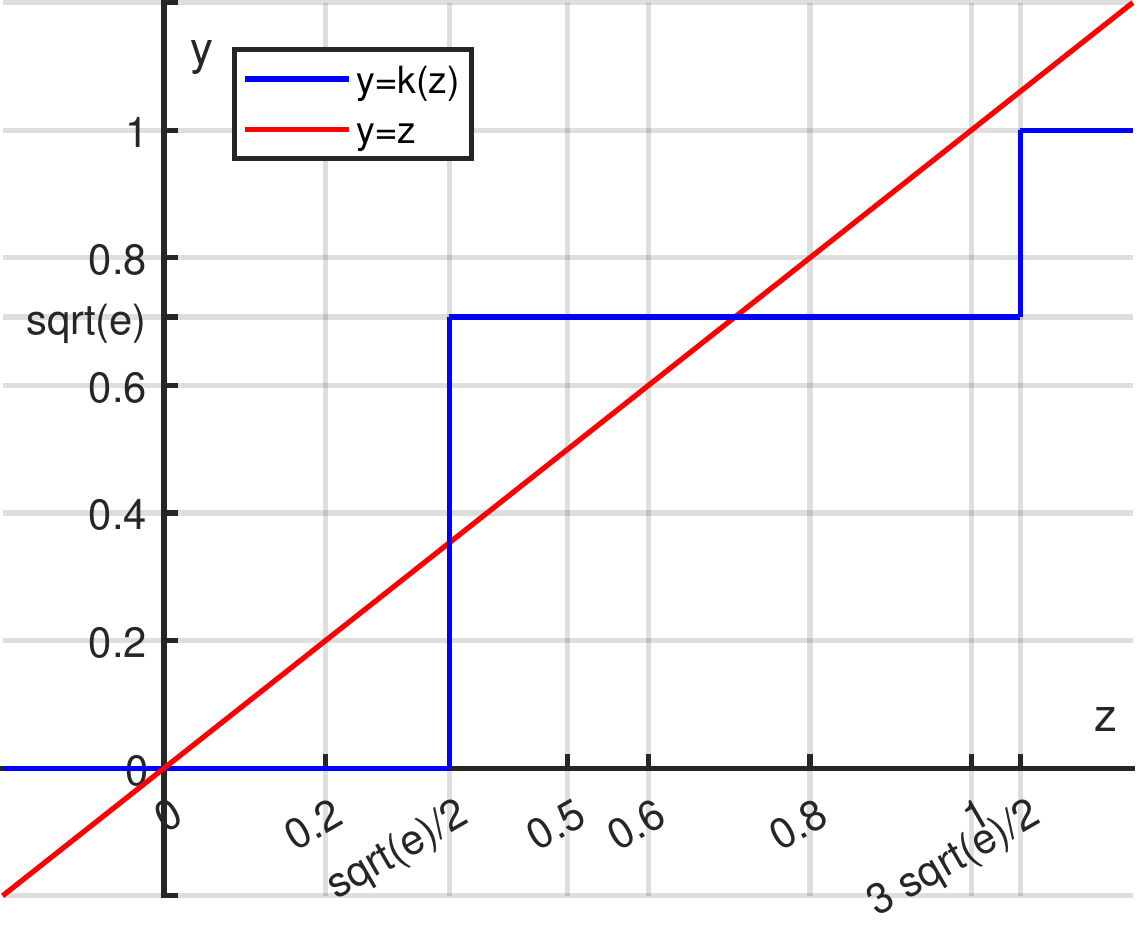}}
\caption{There are only 3 solutions of $k(z)=z$.}
\end{figure}

Of course the function $k(z)$ isn't exactly a step function. However, changing the 
graph of $k(z)$ slightly by having $\beta$ finite and 
making $L^1$-small changes to $d(y)$ can't create intersection points far from where they already are. The only possible values
of $z$ are close to 0, close to $\sqrt{e}/2$ or close to $\sqrt{e}$, as claimed. This completes the first step. 

Let $I_1$ be the union of all the podes where $d(x)$ is close to $\sqrt{e}$, let $I_2$ be the union of all the podes 
where $d(x)$ is close to $\frac12 \sqrt{e}$, and let $I_3$ be the union of all the podes where $d(x)$ is close to 0. By the definition of clique-like, the measure of $I_1$ must be 
close to $\sqrt{e}$, the measure of $I_3$ must be close to $1-\sqrt{e}$ and the measure of 
$I_2$ must be close to 0. 

Note that $d(x)+d(y)$ is above the threshold of $\frac32 \sqrt{e}$ on $I_1 \times I_1$, is close to the threshold on 
$I_1 \times I_2 \cup I_2 \times I_1$, and is below the threshold everywhere else. This implies that $g(x,y)$ is pointwise close to 1 on $I_1 \times I_1$, takes on the average value 
$1/2$ on $I_1 \times I_2$ (in order for the degree function to be $\frac12 \sqrt{e}$) and is
close to zero everywhere else.

Now consider what happens as we vary the size of $I_2$ while keeping $e$ fixed. The entropy associated with the region 
$I_1 \times I_2 \cup I_2 \times I_1$ is linear in the size of $I_2$, as is the extent to which $\tilde t$ (which is the 
variance of the degree function) is reduced from the maximum. That is, $\beta$ must be $O(1)$. However, $\beta$ must diverge as $t$ approaches
the maximum value, as otherwise the graphon would not approach 0 on $I_1\times I_1$ and
1 on $I_3 \times I_3$. This contradiction implies that the size of $I_2$ is in fact zero,
completing the second step.

Next we consider the solutions of $k(z) = z$ near $z=0$ and $z=\sqrt{e}$. Having multiple podes with $d(x)$ close to 0, 
or multiple podes with $d(x)$ close to $\sqrt{e}$, could smear the vertical part of the step function somewhat, but the 
portions of the graph near 0 and $\sqrt{e}$ are nearly flat. (If there were any podes with $d(x)$ close to $\sqrt{e}/2$, 
that would introduce small steps near $z=\sqrt{e}$, insofar as the threshold is $\frac32 \sqrt{e}$, but we just ruled out
the existence of such podes.) Since $k'(z)$ is never greater than 1 near $z=0$ or $z=\frac12 \sqrt{e}$, there can only be one
point near 0 and only one point near $\sqrt{e}$ where $k(z)=z$. 
That is, the graphon must be bipodal, completing the third step.

A bipodal graphon is described by four parameters $(a,b,c,d)$, all between 0 and 1, with 
\begin{equation} 
g(x,y) = \begin{cases} a & x,y < c \cr 
b & x,y > c \cr 
d & x<c<y \hbox{ or } y<c<x\end{cases}
\end{equation}
Since we are looking for clique-like graphons, we want $a \approx 1$, $b \approx 0$, $c \approx \sqrt{e}$, $d \approx 0$. 
We compute the gradient of the edge density, 2star density and entropy with respect to $(a,b,c,d)$ and set 
\begin{equation} \nabla S = \alpha \nabla e + \beta \nabla t. \end{equation}
Those four equations, plus the constraints on $e$ and $t$, give six equations in six unknowns. The system of equations is 
non-degenerate and yields a single family of solutions with $a \approx 1$, $b \approx 0$, $c \approx \sqrt{e}$, and 
$d \approx 0$, namely 
\begin{eqnarray}
a & = & 1 - \delta + O(\delta^2) \cr 
b & = & O(\delta^3) \cr 
c & = & \sqrt{e} + (\frac32 \sqrt{e} -1)\delta + O(\delta^2) \cr 
d & = & \delta + O(\delta^2), 
\end{eqnarray}
where $\delta$ is a small parameter. 
\end{proof}

\section{Above Erd\H{o}s-R\'enyi} \label{sec:small}

We now turn to the bottom of our parameter space, a neighborhood of 
the Erd\H{o}s-R\'enyi curve $t=e^2$, or equivalently $\tilde t=0$. In previous work,
we identified what happened for $e \ne 1/2$ and $\tilde t$ sufficiently small
(where ``sufficiently small'' is $o((e-1/2)^2)$ as $e \to 1/2$). 
We showed that, when $e \ne 1/2$, the optimal graphon is 
bipodal with $c = \frac{\tilde t}{4(e-\frac12)^2} + O(\tilde t^2)$ and 
$d=1-e + O(\tilde t)$. In particular, the degree function is close to $e$ on the large 
pode and $1-e$ on the small pode.  

In this section we bridge the gap between these two regions, proving that the
optimal graphon is unique and bipodal, with parameters that vary smoothly with
$e$ and $\tilde t$, whenever $\tilde t + (e-\frac12)^2$ is small. 

The strategy of proof is a variation of a method we used in \cite{NRS4}
to determine the optimal graphon in the edge-triangle model below the
Erd\H{o}s-R\'enyi curve and when $e \approx 1/2$. 
We begin with an explicit bipodal graphon. 
Using a power-series expansion of the entropy function $H(u)$, we express the
entropy of a graphon in terms of the even moments of $(g(x,y) - \frac12)$. 
By examining the first few moments, we show that an optimal graphon has to be 
close, first in an integral sense and then pointwise, to our model graphon. 
Finally, we use the consistency equation (\ref{eq3}) to show that the optimal
graphon is exactly bipodal and unique. 

\subsection{The ansatz} 

Let 
\begin{equation} 
\zeta = \sqrt{\tilde t + (e - \frac12)^2}, 
\end{equation}
and consider the bipodal graphon $g_0$ with 
\begin{eqnarray}\label{eq:ansatz}
a & = & 1 - e - 2\zeta, \cr 
b & = & 1-e + 2 \zeta, \cr
d & = & 1-e, \cr 
c & = & \frac12 \left (1 - \frac{e-\frac12}{\zeta} \right ). 
\end{eqnarray}
The degree function is exactly $\frac 12 - \zeta$ on the pode of size $c$ and $\frac12 + \zeta$ on the pode of size $1-c$.  

For fixed $e \ne \frac12$, this is the same, to leading order in $\tilde t \ll (e-\frac12)^2$, as what was
previously proven. When $e = \frac12$, this is a symmetric graphon with 
$c=\frac12$ and $b=1-a$. 
Except at $e-\frac12 = \tilde t=0$, the parameters $(a,b,c,d)$ are analytic functions
of $e$ and $\tilde t$. 

\begin{theorem}\label{thm:low-t} For sufficiently small $\zeta$, the entropy-maximizing graphon is 
unique and is well-approximated by the ansatz graphon $g_0$. Specifically, the 
entropy-maximing graphon has 
\begin{eqnarray}\label{what-is-g}
a & = & 1 - e - 2\zeta + O(\zeta^2), \cr 
b & = & 1-e + 2 \zeta + O(\zeta^2), \cr
d & = & 1-e + O(\zeta^2), \cr 
c & = & \frac12 \left (1 - \frac{e-\frac12}{\zeta} \right ) + O(\zeta).
\end{eqnarray}
Furthermore, the exact values of the parameters $a$, $b$, $c$, and $d$ 
are analytic functions of $(e,\tilde t)$ everywhere except at the singular point 
$(\frac12, 0)$. 
\end{theorem}

\begin{corollary} There is an open set in the $(e, \tilde t)$ plane, whose lower boundary
is the entire open line segment $\tilde t=0$, $0<e<1$, on which the optimizing graphon is
bipodal and unique. On this open set, 
the parameters $(a,b,c,d)$ are analytic functions of $(e,\tilde t)$.
\end{corollary}

That is, there is a single bipodal phase just above $\tilde t=0$. This has implications
for the edge-triangle model and for all models where we constrain the density of 
edges and another connected graph $H$
whose vertices all have valence 1 or 2. In \cite{KRRS2}, we
proved results about such models for $e \ne 1/2$ 
by relating the change in the number of $H$'s to
changes in the number of 2stars. A similar approach is promising for 
$e \approx 1/2$. However, the estimates become delicate as $e \to \frac12$, so 
we postpone that analysis to a future work. 

\begin{proof}[Proof of Theorem \ref{thm:low-t}]

Any graphon $g$ with degree function 
\begin{equation}
d(x) = \int_0^1 g(x,y) dy
\end{equation}
can be uniquely written as 
\begin{eqnarray}\label{general-g}
g(x,y) & = & d(x) + d(y) - e + \delta g(x,y) \cr 
& = & (d(x)-\frac12) + (d(y)-\frac12) + (1-e) + \delta g(x,y) \cr 
g(x,y)-\frac12 & = & (d(x)-\frac12) + (d(y)-\frac12) - (e - \frac12) + \delta g(x,y), 
\end{eqnarray}
where $\delta g(x,y)$ is a function with zero marginals: 
\begin{equation}
\int_0^1 g(x,y) \, dx = \int_0^1 g(x,y) \, dy = 0. 
\end{equation}
We will show that $\delta g(x,y)$ is pointwise $O(\zeta^2)$ and that $d(x) - \frac12$ only 
takes on two values, within $O(\zeta^2)$
of $\pm \zeta$. This implies that $g$ is bipodal and follows the estimates 
(\ref{what-is-g}). The analyticity of $(a,b,c,d)$ then follows from the implicit function
theorem. The proof follows 
several steps: 
\begin{enumerate}
\item Using known properties of the entropy function near Erd\H{o}s-R\'enyi to estimate the 
Lagrange multipliers $\alpha$ and $\beta$ that appear in equations (\ref{eq1}--\ref{eq3}).
\item Showing that $g$ is at most tripodal and that the degree function is everywhere 
$\frac12 + o(1)$. This implies that $g$ is pointwise close to $\frac12$.
\item Comparing the entropy of the general graphon $g(x,y)$ of equation (\ref{general-g})
to the entropy of the ansatz graphon $g_0$. This will show that $\iint \delta g(x,y)^2
dx \, dy = O(\zeta^6)$, that $g_0$ comes within $O(\zeta^6)$ of achieving the 
maximum possible entropy, and that the variance of $(d(x)-\frac12)^2$ is $O(\zeta^6)$. 
\item Since $(d(x)-\frac12)^2$ is almost constant (in an $L^2$ sense), and since 
$\int_0^1 (d(x)-\frac12)^2 \, dx = \zeta^2$ (exactly), our graphon must either be 
bipodal with degrees very close to $\frac12 \pm \zeta$, or must be tripodal with a 
very small third pode. We rule out the latter possibility. 
\item We examine the variational equations on the space of bipodal graphons in a 
neighborhood of $g_0$ and show that there is a unique solution that depends analytically
on $(e,\tilde t)$. 
\end{enumerate}

\medskip

\noindent {\bf Step 1:} The function $H(u)$
admits a convergent power series expansion around $u=\frac12$: 
\begin{equation}
H(u) = H\left (\frac12 \right ) + \frac12 \, H''\left ( \frac12 \right ) \left ( u-\frac12
\right )^2 + \frac{1}{24} \, H''''\left ( \frac12 \right ) \left (u-\frac12 \right )^4
+ \cdots.
\end{equation}
This gives rise to a convergent power series expansion for the entropy of a graphon $g$:
\begin{equation}
S(g) = \sum_{k=0}^\infty \frac{\mu_{2k}}{(2k)!} H^{(2k)}\left (\frac12\right ),
\end{equation}
where 
\begin{equation}
\mu_{2k} = \iint \left (g(x,y)-\frac12\right )^{2k} \,dx\,dy.
\end{equation}

When $\tilde t = 0$, the maximal entropy is exactly $H(e)$. As we vary $e$, the infinitesimal change in $t$ is $2e de$, so 
\begin{equation}
\alpha + 2e \beta = H'(e) \approx H''\left (\frac12 \right ) \left (e-\frac12\right ). 
\end{equation} 
The existence of the ansatz graphon $g_0$, with entropy $H(e) + H''(\frac12) \tilde t
+ O(\zeta^4)$, shows that $\beta$ is no less than 
$H''(\frac12) + O(\zeta) = -4+O(\zeta)$. However, if $\beta$ were greater than 
$H''(\frac12) + O(\zeta)$, then there would only be one solution to equation (\ref{eq3}),
namely $z=k(z)=e$. But that gives $\tilde t = 0$. When $\tilde t >0$, we must have 
\begin{equation}
\alpha = -H''(e) + O(\zeta), \qquad \beta = H''(e) + O(\zeta).
\end{equation}

\smallskip

\noindent{\bf Step 2:} With these values of $\alpha$ and $\beta$, the line $y=z$ is 
nearly tangent to $y=k(z)$ at $y=e$. This implies that all solutions to $z=k(z)$ are 
close to $e$, or equivalently close to $1/2$. 

The function $k(z)$ is the convolution of a fixed (scaled) logistic curve. The 
logistic function $1/(1+\exp(-(\alpha+\beta(z+e))))$ has a negative third derivative near $z=e$. Any small convolution
of this function must likewise have a negative third derivative, meaning that its 
second derivative is decreasing and only passes through zero once near $z=e$. 
By Rolle's theorem, this implies 
that a line can only intersect the graph $y=k(z)$ at most three times near $z=e$.
Thus our optimal graphon must be at most tripodal, with degrees close to 1/2. By
equation (\ref{eq1}), this means that the graphon is pointwise close to 1/2, meaning
that $g(x,y) - \frac12$ is $o(1)$ as $\zeta \to 0$. 

\smallskip

\noindent{\bf Step 3:} We now compute the leading terms in the expansions of $S(g_0)$ and $S(g)$. For any graphon, 
let 
\begin{equation}
\nu_k = \int_0^1 \left (d(x) - \frac12 \right)^k \, dx. 
\end{equation}
Note that the first two moments are determined by $e$ and $\tilde t$:
\begin{equation} \nu_1= e - \frac12 \hbox{ and } \nu_2 = \zeta^2 = \tilde t + \left ( e - \frac12 \right )^2. 
\end{equation} 
Higher even moments are bounded from below: 
\begin{equation} 
\nu_{2k} \ge \nu_2^k,
\end{equation}
with equality if and only if $\left( d(x)-\frac12 \right )^2$ is constant. In particular, $\nu_4-\nu_2^2$ is the variance of 
 $\left( d(x)-\frac12 \right )^2$.

 Let 
 \[ \eta = \left ( \iint \delta g(x,y)^2 \, dx], dy \right )^{1/2} \]
 be the $L^2$ norm of $\delta g$. For an arbitrary graphon, the first two non-trivial moments are:
 \begin{eqnarray}
 \mu_2 & = & 2 \nu_2 - \left ( e - \frac12\right )^2 + \eta^2 \cr 
 \mu_4 & = & 2 \nu_4 +  6\nu_2^2 - 12 \left (e-\frac12\right)^2 \nu_2 + 5\left (e-\frac12\right)^4 \cr 
 && + 24 \iint \delta g(x,y) \left [ (d(x)-\frac12)^2(d(y-\frac12)- (e-\frac12) (d(x)-\frac12)(d(y-\frac12)\right ]   dx \, dy \cr 
 && + 12 \iint \delta g(x,y)^2 \left [ (d(x)-\frac12)^2- (d(x)-\frac12)(d(y)-\frac12)\right ]   dx \, dy \cr
  && + 4 \iint \delta g(x,y)^3 \left [ 2(d(x)-\frac12)- (e-\frac12) \right ]   dx \, dy \cr
  && + \iint \delta g(x,y)^4  dx \, dy + 6 (e-\frac12)^2 \|\delta g\|_2^2. 
 \end{eqnarray}
 For the graphon $g_0$, this simplifies to 
  \begin{eqnarray}
 \mu_2 & = & 2 \nu_2 - \left ( e - \frac12\right )^2  \cr 
 \mu_4 & = & 2 \nu_2^2 -12 \nu_2 (e-\frac12)^2 + 5(e-\frac12)^4 \cr 
 & = & 2\zeta^4 + 6 (\zeta^2 - (e-\frac12)^2)^2 - (e-\frac12)^4. 
 \end{eqnarray}
Comparing these, we see that there is an $H''(1/2)\eta^2$ cost in having $\delta g$ nonzero and a $H''''(1/2) (\nu_4-\nu_2^2)/12$ 
cost in having $(d(x) - \frac12)^2$ not be constant. There are also costs in $\mu_4$ from terms proportional to $\delta g^2$ or $\delta g^4$. 

The lowest-order benefits from having $\delta g$ nonzero are terms proportional to 
\begin{eqnarray}
\iint \delta g(x,y) \left (d(x)-\frac12 \right )^2\left ( d(y)-\frac12 \right ) dx\, dy &&  \hbox{and} \cr 
\iint \delta g(x,y) (e - \frac12) \left (d(x)-\frac12 \right )\left ( d(y)-\frac12 \right ) dx\, dy && \end{eqnarray}
By Cauchy-Schwarz, the first term is $O(\zeta \eta \sqrt{\nu_4})$, while the second 
is $O(\zeta^3\eta)$. 
Since $\nu_4$ cannot be much greater than $\nu_2^2$, the maximum possible 
benefit is $O(\zeta^3\eta)$. 

Any benefits from the 
expansion of $\mu_6$ and higher are higher order in $\zeta$, $\eta$, or both, so the total benefit of having 
$\delta g$ nonzero is $O(\zeta^3\eta)$
With a cost proportional to 
$\eta^2$ and benefits that are $O(\zeta^3\eta)$, $\eta$ must itself be $O(\zeta^3)$ and the net benefit of having 
$\delta g$ nonzero is $O(\zeta^6)$. 

This means that the net cost associated with having $(d(x)-\frac12)^2$ differ from a constant must itself be $O(\zeta^6)$. 
There are indeed benefits at higher order, for instance terms proportional to $(e-\frac12) \nu_2\nu_3$ that appear in the
expansion of $\mu_6$, but they are all $O(\zeta^6)$, so the cost in $\mu_4$, proportional to the variance of 
$(d(x)-\frac12)^2$, must be $O(\zeta^6)$. 

\smallskip

\noindent{\bf Step 4:} We recall some facts about cubic polynomials. Suppose that 
\begin{equation}
f(x) = x^3 + a_2 x^2 + a_1 x + a_0 
\end{equation}
has three real roots. The sum of the roots is $-a_2$ and the average of the roots is $-a_2/3$, which is also the unique
point where $f''(x)=0$. Now consider the convolution of $f$ with a distribution of degree functions $d(y)$:
\begin{equation}
\tilde f(x) = \int_0^1 f(x+d(y)) dy.
\end{equation}
If $\tilde f$ has three roots, then the sum of the roots is $-3e-a_2$, where $e = \int_0^1 d(y) dy$ ,and the average of 
the roots is $-e - \frac{a_2}{3}$. Furthermore, a simple algebraic calculations shows that 
\begin{equation} 
\tilde f(-e - \frac{a_2}{3}) = f(-\frac{a_2}{3}) + \int_0^1 (d(y)-e)^3 \, dy.
\end{equation}

Now consider the function 
\[ \frac{1}{1 + \exp(-(\alpha + \beta z)} \]
Near the point of inflection $z=-\alpha/\beta$, this function is approximately cubic, with corrections of order 
$(z + \frac{\alpha}{\beta})^5$. The value of this function at the point of inflection is exactly 1/2. 
Convolving this function by the degree function $d(y)$ we obtain the function $k(z)$. The solutions to $k(z)$ have average
value $\bar d = -e - \frac{\alpha}{\beta}$. The value of $k(\bar d)$ is $\frac12 + \int_0^1 (d(x)-e)^3 \, dx$, plus
$O((\bar d-e)^5)$ corrections due to $k(z)$ not being exactly cubic. 

We have already shown that two of the three roots of $k(z)-z$ are $\frac12 \pm \zeta + O(\zeta^2)$. Since 
$\int_0^1 d(x)^3 \, dx$ is then $O(\zeta^3)$, this implies that the average of the three roots is $\frac12 + O(\zeta^3)$,
which implies that the third root must be $\frac12 + O(\zeta^2)$. Let $c_3$ be the size of this pode. 

We now compute the cost 
\begin{equation} 
\nu_4 - \nu_2^2 = \int_0^1 ((d(x)-\frac12)^2-\zeta^2)^2 \, dx \approx c_3 \zeta^4. 
\end{equation} 
The derivative of the entropy with respect to $c_3$ has a positive term of order $\zeta^4$. Since $d(x)-\frac12$ is pointwise
$O(\zeta)$, all of the contributions to $\mu_6$ and higher have order $\zeta^6$ and higher, and cannot overcome this 
quartic cost. Nor can the cross-terms with $\delta g$, which we have shown to be $O(\zeta^6)$. Since the derivative of the 
entropy with respect to $c_3$ is positive, the entropy is maximized when $c_3=0$. That is, the optimal graphon is 
bipodal, not tripodal, with the degree functions on the two podes being $\frac12 \pm \zeta + O(\zeta^2)$. 

\smallskip

\noindent{\bf Step 5:} Bipodal graphons are described by four parameters $(a,b,c,d)$. The edge density, 2star density, and 
entropy are all analytic functions of $(a,b,c,d)$. Introducing Lagrange multipliers, we obtain four 
analytic equations in six unknowns: 
\begin{equation}
\nabla S  =  \alpha \nabla e + \beta \nabla t 
\end{equation}
This gives a 2-dimensional analytic variety of solutions. To see that $(a,b,c,d)$ are analytic functions of $e$ and $t$,
we need only check that the tangent space does not degenerate. That is, we must check that on this family of solutions, 
the edge and triangle densities can be varied independently to first order. 

However, that is easy. This property obviously holds for the ansatz graphon (\ref{eq:ansatz}), except at the singularity 
$\zeta=0$ (where the derivative of $\zeta$ with respect to $\tilde t$ diverges). 
$\partial d/\partial e=-1$, while $\partial d/\partial \tilde t=0$. However, 
the partial derivatives of $a$, $b$ and $c$ with respect to $\tilde t$ are nonzero, showing
that $\partial_e(a,b,c,d)$ and $\partial_{\tilde t}(a,b,c,d)$ are linearly independent 
vectors. In particular, 
\begin{equation} 
\frac{\partial a}{\partial e} \frac{\partial d}{\partial \tilde t} - 
\frac{\partial d}{\partial e} \frac{\partial a}{\partial \tilde t} = \frac{1}{\zeta}.
\end{equation}

The difference between the true graphon and the ansatz is small, in particular being
$O(\zeta^2)$ for $a$ and $d$, with the derivatives of these terms with respect to 
$\zeta$ being $O(\zeta)$. Since 
\begin{equation}
\frac{\partial \zeta}{\partial e} = \frac{e-\frac12}{\zeta}=O(1) \hbox{ and }
\frac{\partial \zeta}{\partial \tilde t} = \frac{1}{2\zeta},
\end{equation}
these terms can only change $\partial_e(a)$ and $\partial_e(d)$ by $O(\zeta)$ and 
$\partial_{\tilde t}(a)$ and $\partial_{\tilde t}d$ by $O(1)$, resulting in an $O(1)$
change in $\partial_ea \partial_{\tilde t}d - \partial_e d \partial_{\tilde t}a$,
which remains nonzero for small $\zeta$. 

\end{proof}

\section{Bifurcation point(s)} \label{sec:bifurcation}

We now turn our attention to the line $e=\frac12$. When $\tilde t$ is close
to $\left ( \frac{\sqrt{2}-1}{4} \right )$, there are two optimal graphons, one clique-like and the other 
anti-clique-like. When $\tilde t$ is close to 0, there is a unique optimal graphon, 
which must be symmetric under $g \leftrightarrow 1-g$. That is, it must
be bipodal with $c=d=1/2$ and $b=1-a$. 
Somewhere between these regions, there must be a bifurcation point 
$(\frac12, \tilde t^*)$, where the system transitions from having a unique 
optimal maximizer to having multiple inequivalent maximizers. In principle there
might be multiple critical points; we are guaranteed to have at least one. 

We are not prepared to investigate a hypothetical point where a 
graphon that is far from symmetric has a Shannon entropy that matches and then exceeds
the entropy of a symmetric graphon. 
However, we can answer a simpler
question: {\em At what value of $\tilde t$ does the bipodal graphon with $a=1-b$ and $c=d=1/2$ stop being a local maximizer
of the entropy within the 4-dimensional space of bipodal graphons?}

\begin{theorem}\label{thm:bifurcation} There is a critical value $\tilde t^* \approx 0.03727637$ such that
\begin{enumerate} 
\item For all $\tilde t < \tilde t^*$, there is a bipodal graphon with $b=1-a$, $c=d=1/2$, 
that is a local maximizer of the entropy among all bipodal graphons with edge density 
$1/2$ and 2star density $\tilde t + \frac14$.
\item If $\tilde t^* < \tilde t \le 0.0625$, then there exist bipodal graphons with $b=1-a$ and
$c=d=1/2$, but these graphons are not local maximizers. 
\item If $\tilde t > 0.0625$, then there do not exist bipodal graphons
  with $b=1-a$ and $c=d=1/2$.
  
\end{enumerate}
\end{theorem}

\begin{proof}
Every bipodal graphon can be expressed as a linear combination of 
a constant graphon, the function $v(x)+v(y)$, and the function $v(x)v(y)$, where 
\begin{equation}
v(x) = \begin{cases} \sqrt{\frac{1-c}{c}} & x < c, \cr 
- \sqrt{\frac{c}{1-c}} & x>c, 
\end{cases}
\end{equation}
for some constant $c$. For any fixed value of $c$, we can adjust the coefficient of 
$v(x)v(y)$ to maximize the Shannon entropy. This gives us entropy as a function of 
$c$. By doing a power series expansion around $c=1/2$, we determine whether the symmetric
graphon is a local maximum or a local minimum of the entropy. 

We therefore consider graphons with $c = \frac12 + \delta$ and 
\begin{equation}
g(x,y) = \frac12 + \mu [v(x)+ v(y)] + \nu \delta v(x) v(y), 
\end{equation}
We put an explicit factor of $\delta$ in the coefficient of $v(x)v(y)$ in order to make
$\nu$ an even function of $\delta$. Since we only care about the entropy to order $\delta^2$,
the dependence of $\nu$ on $\delta$ will not matter. The degree function is then 
$d(x) = \frac12 + \mu v(x)$, whose variance is $\mu^2$, so $\mu = \sqrt{\tilde t}$. 

Our symmetric graphons have $a=\frac12 + 2\mu$, $b=\frac12 - 2\mu$, and $c=d=\frac12$.
Since $a$ and $b$ must be between 0 and 1, $|\mu|$ cannot be greater than $\frac12$.
These symmetric graphons are only defined when $\tilde t \le \frac{1}{16} = 0.0625$. 

For general (not necessarily symmetric) bipodal graphons, the four parameters $(a,b,c,d)$
are
\begin{eqnarray} 
a & = & \frac12 + 2\mu \sqrt{\frac{1-c}{c}} + \nu \delta \frac{1-c}{c} \cr 
& = & \frac12 + 2\mu + (\nu - 4\mu)\delta + 4(\mu-\nu)\delta^2 + O(\delta^3) \cr 
b & = & \frac12 - 2\mu \sqrt{\frac{c}{1-c}} + \nu \delta \frac{c}{1-c} \cr 
& = & \frac12 - 2\mu + \delta(\nu-4\mu) + 4\delta^2(\nu-\mu) + O(\delta^3) \cr 
c & = & \frac12 + \delta \cr 
d & = & \frac12 + \mu \left ( \sqrt{\frac{1-c}{c}} - \sqrt{\frac{c}{1-c}} \right ) - \nu \delta \cr
& = & \frac12 - (4\mu + \nu) \delta + O(\delta^3).
\end{eqnarray}

We expand the values of $H(a)$, $H(b)$ and $H(d)$ in power series, using the facts that \\
$H'(\frac12 - 2\mu) = -H'(\frac12 + 2\mu)$ and $H''(\frac12 - 2\mu) = H''(\frac12 + 2\mu)$. 
\begin{eqnarray}
H(a) & = & H\left (\frac12 + 2\mu \right ) + H'\left (\frac12+2\mu\right ) \left [ (\nu-4\mu)\delta + 4(\mu-\nu)\delta^2 \right ] \cr && 
+ \frac12 H''\left (\frac12+2\mu\right )(\nu-4\mu)^2 \delta^2 + O(\delta^3), \cr 
H(b) & = & H\left (\frac12 + 2\mu\right ) + H'\left (\frac12+2\mu\right ) \left [(4\mu-\nu)\delta + 4(\mu-\nu)\delta^2 \right ] \cr && 
+ \frac12 H''\left (\frac12+2\mu \right )(\nu-4\mu)^2 \delta^2 + O(\delta^3), \cr 
H(d) & = & H\left (\frac12\right ) + \frac12 H''\left (\frac12 \right ) (\nu+4\mu)^2 \delta^2 + O(\delta^3).\cr
\end{eqnarray}

The Shannon entropy of a bipodal graphon is 
\begin{eqnarray}
S(g) &=& c^2 H(a) + (1-c)^2 H(b) + 2c(1-c) H(d) \cr 
& = & \frac14[H(a)+H(b)+2H(d)] + \delta [H(a)-H(b)] + \delta^2[H(a)+H(b)-2H(d)]. 
\end{eqnarray}
Plugging in our previously computed values of $H(a)$, $H(b)$ and $H(d)$ gives
\begin{eqnarray}\label{S-is}
S(g) & = & \frac12 \left ( H(\frac12 + 2\mu) + H(\frac12) \right ) 
- 6 \delta^2 \mu H'(\frac12 + 2\mu)  \cr 
&& + \frac14 \delta^2 \left [ (\nu - 4\mu)^2 H''(\frac12 + 2\mu) + (\nu+4\mu)^2 H''(\frac12)\right ] \cr 
&& + 2\delta^2 \left ( H(\frac12 + 2\mu) - H(\frac12) \right ) + O(\delta^3).
\end{eqnarray}
That is, the change in the entropy from the symmetric graphon with $\delta=0$ is 
proportional to $\delta^2$ (plus higher-order terms). To leading order,
the change in entropy is quadratic in $\nu$. This quadratic function is maximized when 
\begin{equation}\label{nu-is} 
\nu = 4\mu \frac{H''(\frac12+2\mu) - H''(\frac12)}{H''(\frac12+2\mu) + H''(\frac12)}. 
\end{equation}

Next we compute $H$, $H'$ and $H''$ at $\frac12$ and $\frac12 + 2\mu$: 
\begin{eqnarray} 
H(\frac12) = \frac12 \ln(2) && H(\frac12 + 2\mu) = -\frac12 \left ( 
(\frac12 + 2\mu) \ln(\frac12 + 2\mu) + (\frac12 - 2\mu) \ln (\frac12 - 2\mu) \right )\cr 
H'(\frac12) = 0 && H'(\frac12 + 2\mu) = \frac12 (\ln(\frac12 - 2\mu) - \ln(\frac12 + 2\mu)) \cr 
H''(\frac12) = -2 && H''(\frac12 + 2\mu) = \frac{-2}{1-16\mu^2} 
\end{eqnarray}
The combinations that appear in equation (\ref{S-is}) are: 
\begin{eqnarray}
H''(\frac12 + 2\mu) - H''(\frac12) & = \displaystyle{-2 \left ( \frac{1}{1-16\mu^2} - 1 \right )} & = 
-2 \left (\frac{16\mu^2}{1-16\mu^2} \right ) \cr 
H''(\frac12 + 2\mu) + H''(\frac12) & = \displaystyle{-2 \left ( \frac{1}{1-16\mu^2} + 1 \right )} & = 
-2 \left (\frac{2-16\mu^2}{1-16\mu^2} \right ) \cr
\nu & = \displaystyle{4\mu \left( \frac{16\mu^2}{2-16\mu^2} \right )} & = 4 \mu \left (\frac{8\mu^2}{1-8\mu^2}
\right ) \cr 
\nu + 4\mu & = \displaystyle{4\mu \left (\frac{8\mu^2}{1-8\mu^2} + 1 \right )} & = \frac{4\mu}{1-8\mu^2} \cr 
\nu - 4\mu & = \displaystyle{4\mu \left ( \frac{8\mu^2}{1-8\mu^2} - 1 \right )} & = 
4\mu \left ( \frac{16 \mu^2 - 1}{1-8\mu^2} \right ).
\end{eqnarray}
Combining, we have that 
\begin{eqnarray} 
\frac{\Delta S}{\delta^2} & = & 3 \mu \ln \left ( \frac{\frac12 + 2\mu}{\frac12 - 2\mu}
\right ) - \frac{2}{1-16\mu^2} \frac{16\mu^2(16\mu^2-1)^2}{4(1-8\mu^2)^2} \cr 
&& - \frac24 \frac{16\mu^2}{(1-8\mu^2)^2} - \left ( 
(\frac12 + 2\mu) \ln(\frac12 + 2\mu) + (\frac12 - 2\mu) \ln(\frac12  2\mu) + \ln(2) \right )
+ O(\delta), 
\end{eqnarray}
where the first term is $-6\mu H'(\frac12+2\mu)$, the second is $\frac14 (\nu-4\mu)^2 H''(\frac12+2\mu)$, the third is $\frac14 (\nu+4\mu)^2 H''(\frac12)$, and the last is 
$2(H(\frac12+2\mu) -H(\frac12))$. The logarithmic terms simplify to 
\[
\mu \ln(1+4\mu) - \mu \ln(1-4\mu) - \frac12 \ln(1-16\mu^2),
\]
while the algebraic terms simplify to 
\[ \frac{-16\mu^2}{1-8\mu^2}. \]
The total is negative when $\mu$ is small, going as $-64 \mu^4/3$, but turns positive for
larger values of $\mu$, diverging logarithmically as $\mu$ approaches 1/4. 
The crossover point is at 
\begin{equation}\label{eqnCriticalPoint}
  \mu^* \approx 0.1930708944, \qquad \tilde t^* \approx 0.0372763703.
  \end{equation}
\end{proof}

\section*{Appendix}

We include here proofs of two key steps in the project which began
with \cite{RS1}; the existence
of the Boltzmann entropy, Theorem \ref{thmBoltzmann} (proven is less
generality in \cite{RS1,RS2}), and the
connection with large finite graphs, Theorem \ref{thmTypicality}.

\begin{proof}[Proof of Theorem \ref{thmBoltzmann}] We first prove that 
$B_{t_1,\ldots,t_k}$ is well-defined.
A priori we only know that $\liminf\ln
  (Z_{t_1,\ldots,t_k}^{n,\delta})/n^2$ and $\limsup\ln (Z_{t_1,\ldots,t_k}^{n,\delta})/n^2$
  exist as ${n\to\infty}$. However, we will show that they both
  approach $\max S(g)$ as $\delta \to 0^+$.

  We need to define a few sets. Let $U_\delta$ be the set of graphons
  $g$ with each $\tdens_i(g)$ strictly within $\delta$ of $t_i$,
  i.e.\ the preimage of an open $k$-cube of side $2\delta$ in
  $t$-space, let $F_\delta$ be the preimage of the closed
  $k$-cube. and let $\tilde U_\delta^n$ and $\tilde F_\delta^n$ be the corresponding sets 
  in $\widetilde \www$. Let $|U_\delta^n|$ and
  $|F_\delta^n|$ denote the number of graphs with $n$ vertices whose
  checkerboard graphons lie in $U_\delta$ or
  $F_\delta$. By the large deviations principle, Theorem \ref{thmCV},

  \begin{equation} \limsup_{n \to \infty} \frac{\ln|F_\delta^n|}{n^2}
    \le \sup_{\tilde g\in \tilde F_\delta} S(\tilde g),\end{equation}
  which also equals $\sup_{g \in F_\delta} S(g)$, and 
  \begin{equation} \liminf_{n \to \infty} \frac{\ln|U_\delta^n|}{n^2}
    \ge \sup_{\tilde g \in \tilde U_\delta} S(\tilde g),\end{equation}
  which also equals $\sup_{g \in U_\delta} S(g)$.  This yields a
  chain of inequalities
  \begin{equation} \sup_{U_\delta} S(g) \le \liminf
    \frac{\ln|U^n_\delta|}{n^2} \le \limsup
    \frac{\ln|U^n_\delta|}{n^2} \le \limsup
    \frac{\ln|F_\delta^n|}{n^2} \le  \sup_{F_\delta} S(g) \le 
    \sup_{U_{\delta+\delta^2}} S(g).\end{equation} As $\delta \to 0^+$,
  the limits of $\sup_{U_\delta} S(g)$ and
  $\sup_{U_{\delta+\delta^2}} S(g)$ are the same, and everything in
  between is trapped.

So far we have proven that $\displaystyle B_{t_1,\ldots,t_k}$ exists and equals
\begin{equation} \lim_{\delta \to 0^+}
\sup_{g \in U_\delta} S(g).\end{equation} 
This limit is manifestly at least as big as $\max S(g)$, the maximum value of $S(g)$ among 
graphons with each $\tdens_i$ exactly equal to $t_i$. To see that is cannot be greater,
imagine a sequence of graphons $g_j$  with each $\tdens_i(g_k)$ converging to $t_i$, 
and with $S(g)$ greater than $\max S(g) + \epsilon$. By the compactness of $\widetilde 
\www$, there is a subsequence whose classes in $\widetilde \www$ converge to that of 
a graphon $g_\infty$. The densities
$\tdens_i$ are continuous in the cut metric, so $\tdens_i(g_\infty)=t_i$. The 
entropy functional 
is upper-semicontinuous \cite{CV}, so $S(g_\infty) \ge \max S(g) + \epsilon$, 
which contradicts the definition of $\max S(g)$. 

\end{proof} 

\medskip

\begin{proof}[Proof of Theorem \ref{thmTypicality}]
Let $U_\epsilon$ denote the open set in 
$\widetilde \www$ of graphons whose cut distance 
from $g_0$ is strictly less than $\epsilon$, and let $\bar U_\epsilon$ be those
graphons of distance $\epsilon$ or less. The complements $U_\epsilon^c$ 
(resp.~$\bar U_\epsilon^c$) are then closed (resp.~open) sets of graphons whose 
distance from $g_0$
is greater than or equal to $\epsilon$ (resp.~strictly greater than $\epsilon$). 
Let $V_\delta$ (resp.~$\bar V_\delta$) denote the set of all graphons $g$ with densities $\tdens_i(g)$
in $(t_i-\delta, i_i+\delta)$ (resp. $[t_i-\delta, t_i+\delta]$) for each $i$.

If $V_\delta \cap \bar U_\epsilon^c$ is empty for any $\delta$, then {\em all\/} 
checkerboard graphons in $V_\delta$ are close to $g_0$ and there is nothing left to prove. 
Otherwise, let $S_0 = S(g_0)$. Let
$S_{3,\delta,\epsilon}$
be the supremum of $S(g)$ on $V_\delta \cap \bar U_\epsilon^c$.
For fixed $\epsilon$, let $S_3 = \lim_{\delta \to 0} S_{3,\delta,\epsilon}$.

The proof proceeds in five steps:
\begin{enumerate}
\item For any fixed $\epsilon$, showing that
$S_3  < S_0$.
\item Picking $K < S_0 -S_3$ and numbers
$S_1$ and $S_2$, such that $S_0 > S_1 > S_2 > S_3$ and $S_1 = S_2 + K$.
\item Showing that, for any $\delta$, the number of graphs in
$\mathcal{G}_{\delta,n}$
is eventually greater than $\exp(S_1 n^2)$.
\item Showing that, for $\delta$ sufficiently small, the number of graphs in
$\mathcal{G}_{\delta,n} \cap U_\epsilon^c$ is eventually smaller than
$\exp(S_2 n^2)$.
\item Concluding that, for $\delta$ sufficiently small and $n$ larger than
a number
that depends only on $\delta$ and $\epsilon$, the number of graphs in
$\mathcal{G}_{\delta,n} \cap U_\epsilon^c$ divided by the number of
graphs in
$\mathcal{G}_{\delta,n}$ is less than $\exp(-Kn^2)$.
\end{enumerate}

\medskip

\noindent{\bf Step 1:} Suppose that $S_3 \ge S_0$. Then there would
exist a sequence
of graphons $g_1, g_2, \ldots$ in $U_\epsilon^c$, with densities
approaching $(t_1,\ldots,t_k)$, with $\limsup S(g_j)\ge S_0$.
As in the proof of Theorem \ref{thmBoltzmann}, we use the compactness of 
$\widetilde \www$, the continuity of $\tdens_i$, and the semi-continuity 
of $S$ to construct
a subsequential limit $g_\infty \in \bar U_\epsilon^c$ with 
densities equal to $(t_1,\ldots,t_k)$ and
with $S(g_\infty) \ge S_0$. That contradicts the uniqueness of $g_0$, so we 
conclude that $S_3 < S_0$.

Note that $S_3$ cannot be negative, as the functional $S(g)$ is 
positive semi-definite. So what happens when $S_0=0$? In that case, 
$V_\delta \cap \bar U_\epsilon^c$ must be empty when $\delta$ is small. 

\medskip

\noindent{\bf Step 2:} Take $K = (S_0-S_3)/3$, $S_1 = S_0-K$ and $S_2=S_0-2K$.

\medskip

\noindent{\bf Step 3:} For any $\delta>0$, the supremum of $S(g)$ over 
$V_\delta$ is at least
$S_0$, and so is strictly greater than $S_1$. Since $V_\delta$ is 
an open set, 
\begin{equation} \liminf_{n\to \infty}
\frac{1}{n^2} \ln(\# {\mathcal{G}}_{n,\delta}) \ge S_0 > S_1,
\end{equation} 
so for all sufficiently large values of $n$, $\# \mathcal{G}_{n,\delta} > \exp(S_1
n^2)$.

\medskip

\noindent{\bf Step 4:} Since $\lim_{\delta \to 0} S_{3,\delta,\epsilon}
= S_3 < S_2$,
there exists a nonzero value of $\delta$ for which
$S_{3,\delta+\delta^2,\epsilon} < S_2$. The number of graphs in 
$V_{\delta} \cap U_\epsilon^c$ 
is bounded by the number of graphs in the closed set $\bar V_\delta \cap U_\epsilon^c$
and the entropy $S(g)$ on $\bar V_\delta \cap U_\epsilon^c$ is bounded by 
$S_{3,\delta+\delta^2,\epsilon} < S_2$. 
By the first half of Theorem \ref{thmCV},  
\begin{equation} \limsup_{n\to \infty}
\frac{1}{n^2} \ln(\# (\bar V_\delta \cap U_\epsilon^c)) <S_2,
\end{equation} 
so the smaller quantity $\#(V_\delta \cap U_\epsilon^c)$ grows strictly
slower than $\exp(n^2 S_2)$,
and in particular is eventually bounded by $\exp(n^2 S_2)$.

\medskip

\noindent{\bf Step 5:} Now we consider the order of operations. Given
$\epsilon$, we first compute $S_3$ and define $k$, $S_1$ and $S_2$. We then pick a $\delta$
such that the size of $V_{\delta} \cap U_\epsilon^c$ is bounded by $\exp(n^2 S_2)$
for all sufficiently large $n$. The phrase ``sufficiently large'' means that
there is a number $N_1$, depending on $\delta$ and $\epsilon$, such that the bound
applies for all $n>N_1$. Meanwhile, the number of graphs in $\mathcal{G}_{n,\delta}$ 
is at least $\exp(n^2 S_1)$ for all $n$ greater than another constant $N_2$. 
Pick $N = \max(N_1, N_2)$.

The upshot is that for this value of $\delta$, and for all $n>N$,
\[ \frac{\#(\mathcal{G}_{n,\delta} \cap
U_\epsilon^c)}{\#(\mathcal{G}_{n,\delta})}
\le \frac{\exp(n^2 S_2)}{\exp(n^2 S_1)} = \exp(-Kn^2). \]

\end{proof}

\end{document}